\documentclass[12pt,reqno]{amsart}
\usepackage[top=2cm,bottom=2cm,right=2.5cm,left=2.5cm]{geometry}
\usepackage{amssymb}
\usepackage{amsmath, amsthm, amscd, amsfonts, amssymb, graphicx, color}
\usepackage[bookmarksnumbered, colorlinks, plainpages]{hyperref}

\usepackage{hyperref}

\newtheorem{theorem}{Theorem}[section]

\newtheorem{proposition}[theorem]{Proposition}
\newtheorem{corollary}[theorem]{Corollary}
\theoremstyle{definition}
\newtheorem{definition}[theorem]{Definition}
\newtheorem{example}[theorem]{Example}

\theoremstyle{remark}
\newtheorem{remark}[theorem]{Remark}

\numberwithin{equation}{section}

\begin{document}
	
	\setcounter{page}{1}
	
	\title[$K$-biframes in Hilbert spaces]{$K$-biframes in Hilbert spaces}

	\author[A. Karara, M. Rossafi]{Abdelilah Karara$^{1}$ and Mohamed Rossafi$^{2*}$}
	
	\address{$^{1}$Department of Mathematics Faculty of Sciences, University of Ibn Tofail, B.P. 133, Kenitra, Morocco}
	\email{\textcolor[rgb]{0.00,0.00,0.84}{abdelilah.karara@uit.ac.ma}}
	\address{$^{2}$Department of Mathematics Faculty of Sciences, Dhar El Mahraz University Sidi Mohamed Ben Abdellah, Fez, Morocco}
	\email{\textcolor[rgb]{0.00,0.00,0.84}{mohamed.rossafi@usmba.ac.ma}}
	\date{
		\newline \indent $^{*}$ Corresponding author}
	
	
	\begin{abstract}
In this paper, we introduce a new concept of $K$-biframes for Hilbert spaces. We then examine several characterizations with the assistance of a biframe operator. Moreover, we investigate their properties from the perspective of operator theory by establishing various relationships and properties.

\end{abstract}
\maketitle

\baselineskip=12.4pt

\section{Introduction}
 
\smallskip\hspace{.6 cm} The notion of frames in Hilbert spaces was introduced by Duffin and Schaffer \cite{Duffin} in 1952 to research certain difficult nonharmonic Fourier series problems. Following the fundamental paper \cite{Daubechies} by Daubechies, Grossman, and Meyer, frame theory started to become popular, especially in the more specific context of Gabor frames and wavelet frames \cite{Gabor}. For more detailed information on frame theory, readers are recommended to consult: \cite{Christensen, Daubechies2, Han}.

The concept of $K$-frames was introduced by Laura Găvruţa, serve as a tool for investigating atomic systems with respect to a bounded linear operator $K$ in a separable Hilbert space. $K$-frames generalize ordinary frames by requiring that the lower frame bound is applicable only to elements within the range of $K$.

The concept of pair frames, referring to a pair of sequences in a Hilbert space, was first presented in \cite{Fer} by Azandaryani and Fereydooni. Parizi, Alijani and Dehghan \cite{MF} studied Biframes, which generalize controlled frames in Hilbert space. The concept of a frame is defined by a single sequence, but to define a biframe we will need two sequences. In fact, the concept of  a biframe is both a generalization of controlled frames and a special case of pair frames.

 In this paper, we introduce the concept of $K$-biframes in Hilbert space and present several examples of this type of frame. Moreover, we investigate a characterization of $K$-biframe  using the biframe operator. Finally, in our exploration of biframes, we investigate their characteristics from the perspective of operator theory by establishing various properties.

\section{Notation and preliminaries}
Throughout this paper, $\mathcal{H}$ represents a separable Hilbert space. The notation $\mathcal{B}(\mathcal{H},\mathcal{K})$ denotes the collection of all bounded linear operators from $\mathcal{H}$ to the Hilbert space $\mathcal{K}$. When $\mathcal{H} = \mathcal{K}$, this set is denoted simply as $\mathcal{B}(\mathcal{H})$. We will use $\mathcal{N}(\mathcal{T})$ and $\mathcal{R}(\mathcal{T})$ for the null  and range space of an operator $\mathcal{T}\in \mathcal{B}(\mathcal{H})$. Additionally, $\mathrm{GL}(\mathcal{H})$ represents the collection of all invertible, bounded linear operators acting on $\mathcal{H}$.

Certainly, let's begin with some preliminaries. Before  delving into the details, we briefly recall the definitions of a frame and $K$-frame:

\begin{definition}
A sequence $\left\{x_i\right\}_{i=1}^{\infty}$ in $\mathcal{H}$ is called a frame for $\mathcal{H}$ if there exist two constants $0<A \leq B<\infty$ such that
$$
A\left\| x\right\|^2 \leq \sum_{i=1}^{\infty}\vert\left\langle x, x_i\right\rangle\vert^{2}\leq B\left\| x\right\|^2, \text { for all } x \in \mathcal{H} .
$$

\end{definition}
\begin{definition}
Let $K \in \mathcal{B}(\mathcal{H})$. A sequence $\left\{x_i\right\}_{i=1}^{\infty}$ in $\mathcal{H}$ is called a $K$-frame for $\mathcal{H}$ if there exist two constants $0<\lambda \leq \mu<\infty$ such that
$$
\lambda\left\|K^{\ast} x\right\|^2 \leq \sum_{i=1}^{\infty}\vert\left\langle x, x_i\right\rangle\vert^{2} \leq \mu\|x\|^2, \text { for all } x \in \mathcal{H} .
$$

\end{definition}
\begin{theorem}\cite{Abramovich}
$\mathcal{T} \in\mathcal{B}(\mathcal{H})$ is an injective and closed range operator if and only if there exists a constant $c>0$ such that $c\|x\|^2 \leq\|\mathcal{T}  x\|^2$, for all $x \in \mathcal{H}$
\end{theorem}
\begin{definition}\cite{Limaye} Let $\mathcal{H}$ be a Hilbert space, and suppose that $\mathcal{T} \in \mathcal{B}(\mathcal{H})$ has a closed range. Then there exists an operator $\mathcal{T}^{+} \in \mathcal{B}(\mathcal{H})$ for which
$$
N\left(\mathcal{T}^{+}\right)=\mathcal{R}(\mathcal{T})^{\perp}, \quad R\left(\mathcal{T}^{+}\right)=N(\mathcal{T})^{\perp}, \quad \mathcal{T} \mathcal{T}^{+} x=x, \quad x \in \mathcal{R}(\mathcal{T}) .
$$

We call the operator $\mathcal{T}^{+}$ the pseudo-inverse of $\mathcal{T}$. This operator is uniquely determined by these properties. In fact, if $\mathcal{T}$ is invertible, then we have $\mathcal{T}^{-1}=\mathcal{T}^{+}$.
\end{definition}

\begin{theorem}\cite{Douglas} \label{Doglas th} 
 Let $\mathcal{H}$ be a Hilbert space and $\mathcal{T}_{1},\mathcal{T}_{2}\in\mathcal{B}(\mathcal{H})$. The following statements are equivalent:
 \begin{enumerate}
 \item $\mathcal{R}(\mathcal{T}_{1})\subset\mathcal{R}(\mathcal{T}_{2})$
 \item  $\mathcal{T}_{1} \mathcal{T}_{1}^{\ast} \leq \lambda^2 \mathcal{T}_{2}\mathcal{T}_{2}^{\ast}$ for some $\lambda \geq 0$;
 \item $\mathcal{T}_{1} = \mathcal{T}_{2} U$ for some $U\in\mathcal{B}(\mathcal{H})$.
 
 \end{enumerate}
\end{theorem}

\begin{definition}\cite{MF}
Let $(\mathcal{X}, \mathcal{Y})=\left(\left\{x_i\right\}_{i=1}^{\infty},\left\{y_i\right\}_{i=1}^{\infty}\right)$ be a biframe for $\mathcal{H}$. We define the biframe operator $ S_{\mathcal{X}, \mathcal{Y}} $ as follows:
$$
S_{\mathcal{X}, \mathcal{Y}}: \mathcal{H} \longrightarrow \mathcal{H}, \quad S_{\mathcal{X}, \mathcal{Y}}(x):=\sum_{i=1}^{\infty}\left\langle x, x_i\right\rangle y_i .
$$ 
\end{definition}

\section{$K$-biframes in Hilbert spaces}
\begin{definition}\label{k_biframe}
Let $K \in \mathcal{B}(\mathcal{H})$. A pair $(\mathcal{X}, \mathcal{Y})_{K}=\left(\left\{x_i\right\}_{i=1}^{\infty},\left\{y_i\right\}_{i=1}^{\infty}\right)$ of sequences in $\mathcal{H}$ is called a $K$-biframe for $\mathcal{H}$, if there exist two constants $0<A \leq B<\infty$ such that
$$
A\left\|K^{\ast} x\right\|^2 \leq \sum_{i=1}^{\infty}\left\langle x, x_i\right\rangle \left\langle y_i, x\right\rangle \leq B\|x\|^2, \text { for all } x \in \mathcal{H} .
$$
\end{definition}
The numbers $A$ and $B$ are called respectively the lower and upper bounds for the $K$-biframe $(\mathcal{X}, \mathcal{Y})_{K}$ respectively. 
 If $K$ is equal to $\mathcal{I}_{\mathcal{H}}$, the identity operator on $\mathcal{H}$, then $K$-biframes is biframes.

\begin{remark}
According to Definition \ref{k_biframe}, the following statements are true for a sequence $\mathcal{X}=\left\{x_k\right\}_{k=1}^{\infty}$ in $\mathcal{H}$ :
\begin{enumerate}
\item If $(\mathcal{X}, \mathcal{X})$ is a $K$-biframe for $\mathcal{H}$, then $\mathcal{X}$ is a K-frame for $\mathcal{H}$.
\item If $(\mathcal{X}, C \mathcal{X})$ is a $K$-biframe for some $C \in \mathrm{GL}(\mathcal{H})$, then $\mathcal{X}$ is a $C$-controlled K-frame for $\mathcal{H}$.
\item If $(C_{1} \mathcal{X}, C_{2} \mathcal{X})$ is a $K$-biframe for some $C_{1}$ and $C_{2}$ in $\mathrm{GL}(\mathcal{H})$, then $\mathcal{X}$ is a $(C_{1}, C_{2})$-controlled K-frame for $\mathcal{H}$.
\end{enumerate}
\end{remark}
\begin{example} Let $\mathcal{H}=\mathbb{C}^4$ and $\left\{e_1, e_2, e_3, e_4\right\}$ be an orthonormal basis for $\mathcal{H}$, Define $$K: \mathcal{H} \rightarrow \mathcal{H}\quad \text {by} \quad K e_1=e_1,\; K e_2=e_1,\; K e_3=2e_2,\; K e_4=3e_3.$$ 

 We consider two sequences $\mathcal{X}=\left\{x_i\right\}_{i=1}^{4}$ and $\mathcal{Y}=\left\{y_i\right\}_{i=1}^{4}$ defined as follows: 
 $$\mathcal{X}=\left\{e_1, e_1, 2e_2,3e_3\right\}$$
 And
  $$\mathcal{Y}=\left\{e_1, e_1, e_2,e_3\right\}.$$
  For $x\in \mathcal{H}$ we have,
  \begin{align*}
\sum_{i=1}^{4}\left\langle x, x_i\right\rangle\left\langle y_i, x\right\rangle&=\left\langle x, e_1\right\rangle\left\langle e_1, x\right\rangle+ \left\langle x, e_1\right\rangle\left\langle e_1, x\right\rangle+ \left\langle x, 2e_2\right\rangle\left\langle e_2, x\right\rangle + \left\langle x, 3e_3\right\rangle\left\langle e_3, x\right\rangle\\
&=2\left\langle x, e_1\right\rangle\left\langle e_1, x\right\rangle+2\left\langle x, e_2\right\rangle\left\langle e_2, x\right\rangle + 3\left\langle x, e_3\right\rangle\left\langle e_3, x\right\rangle \\
&\leq 3 \Vert f \Vert^{2}.
\end{align*}
On the other hand, we have
 $$\mathcal{X}=\left\{Ke_1, Ke_2, Ke_3,Ke_4\right\},$$
 And
  $$\mathcal{Y}=\left\{Ke_1, Ke_2, 2Ke_2,3Ke_3\right\}.$$
For $x\in \mathcal{H}$ we have,
  \begin{align*}
&\sum_{i=1}^{4}\left\langle x, x_i\right\rangle\left\langle y_i, x\right\rangle=\left\langle x, Ke_1\right\rangle\left\langle Ke_1, x\right\rangle+ \left\langle x, Ke_2\right\rangle\left\langle Ke_2, x\right\rangle+ \left\langle x, Ke_3\right\rangle\left\langle 2Ke_3, x\right\rangle + \left\langle x, Ke_4\right\rangle\left\langle 3Ke_4, x\right\rangle\\
&=\left\langle K^{\ast}x, e_1\right\rangle\left\langle e_1, K^{\ast}x\right\rangle+ \left\langle K^{\ast}x, e_2\right\rangle\left\langle e_2,K^{\ast} x\right\rangle+ 2\left\langle K^{\ast}x, e_3\right\rangle\left\langle e_3, K^{\ast}x\right\rangle + 3\left\langle K^{\ast}x, e_4\right\rangle\left\langle e_4,K^{\ast} x\right\rangle\\
&\geq \Vert K^{\ast}x \Vert
\end{align*}
Therefore $(\mathcal{X}, \mathcal{Y})_{K}$ is $K$-biframe with bounds $1$ and $3$
\end{example}
\begin{definition}
Let $K \in \mathcal{B}(\mathcal{H})$. A pair  $(\mathcal{X}, \mathcal{Y})_{K}=\left(\left\{x_i\right\}_{i=1}^{\infty},\left\{y_i\right\}_{i=1}^{\infty}\right)$ of sequences in $\mathcal{H}$ is said to be a tight $K$-biframe with bound $A$ if
$$
A\left\|K^* x\right\|^2=\sum_{i=1}^{\infty}\left\langle x, x_i\right\rangle\left\langle y_i, x\right\rangle, \text { for all } x \in \mathcal{H} .
$$

When $A=1$, it is called a Parseval $K$-biframe.
\end{definition}
\begin{example}
let $\lbrace e_{i}\rbrace_{i=1}^{\infty}$ be an orthonormal basis for $\mathcal{H}$,  Define $$K: \mathcal{H} \rightarrow \mathcal{H}\quad \text {by} \quad K e_i=e_i,\quad\text{for}\;\; i=1,2,3,\ldots$$

 Consider the sequences $\mathcal{X}=\left\{x_i\right\}_{i=1}^{\infty}$ and $\mathcal{Y}=\left\{y_i\right\}_{i=1}^{\infty}$ defined as follows:
$$\mathcal{X}=\left\{e_1, 2e_2, 3e_3, 4e_4, \ldots\right\}, $$ and
$$\mathcal{Y}=\left\{ e_1,\dfrac{1}{2} e_2, \dfrac{1}{3}e_3, \dfrac{1}{4}e_4,\ldots\right\} .$$
For $x \in \mathcal{H}$, we have
$$
\begin{aligned}
\sum_{i=1}^{\infty}\left\langle x, x_i\right\rangle\left\langle y_i, x\right\rangle&=\left\langle x, e_1\right\rangle\left\langle e_1, x\right\rangle+\left\langle x, 2e_2\right\rangle\left\langle \dfrac{1}{2}e_2, x\right\rangle+\left\langle x, 3e_3\right\rangle\left\langle \dfrac{1}{3}e_3, x\right\rangle\cdots\\
&= \sum_{i=1}^{\infty}\left\langle x, Ke_i\right\rangle\left\langle Ke_i, x\right\rangle\\ 
&=\sum_{i=1}^{\infty}\left|\left\langle K^*x, e_i\right\rangle\right|^2\\
&=\left\|K^* x\right\|^2
\end{aligned}
$$
Therefore $\left(\left\{x_i\right\}_{i=1}^{\infty},\left\{y_i\right\}_{i=1}^{\infty}\right)$ is a Parseval $K$-biframe,
\end{example}

\begin{theorem}
  $(\mathcal{X}, \mathcal{Y})_{K}=\left(\left\{x_i\right\}_{i=1}^{\infty},\left\{y_i\right\}_{i=1}^{\infty}\right)$ is a $K$-biframe if and only if $(\mathcal{Y}, \mathcal{X})_{K}=\left(\left\{y_i\right\}_{i=1}^{\infty},\left\{x_i\right\}_{i=1}^{\infty}\right)$ is a $K$-biframe.
\end{theorem}
\begin{proof}
Let $(\mathcal{X}, \mathcal{Y})_{K}=\left(\left\{x_i\right\}_{i=1}^{\infty},\left\{y_i\right\}_{i=1}^{\infty}\right)$ be a $K$-biframe with bounds $A$ and $B$. Then, for every $x \in \mathcal{H}$,
$$
A\left\|K^{\ast} x\right\|^2\leq \sum_{i=1}^{\infty}\left\langle x, x_i\right\rangle\left\langle y_i, x\right\rangle \leq B\left\| x\right\|^2.
$$

Now, we can write
$$
\begin{aligned}
\sum_{i=1}^{\infty}\left\langle x, x_i\right\rangle\left\langle y_i, x\right\rangle&=\overline{\sum_{i=1}^{\infty}\left\langle x, x_i\right\rangle\left\langle y_i, x\right\rangle}\\
&=\sum_{i=1}^{\infty}\overline{\left\langle x, x_i\right\rangle\left\langle y_i, x\right\rangle}\\ &=\sum_{i=1}^{\infty}\left\langle x, y_i\right\rangle\left\langle x_i, x\right\rangle .
\end{aligned}
$$
Therefore
$$
A\left\|K^{\ast} x\right\|^2 \leq \sum_{i=1}^{\infty}\left\langle x, y_i\right\rangle\left\langle x_i, x\right\rangle \leq B\left\| x\right\|^2 .
$$

This implies that, $(\mathcal{Y}, \mathcal{X})_{K}$ is a $K$-biframe with bounds $A$ and $B$. The reverse of this statement can be proved similarly..
\end{proof}
 \begin{proposition}
Let $K\in \mathcal{B}(\mathcal{H})$ and $\lbrace z_{i}\rbrace_{i=1}^{\infty}$ be a $K$-frame for $\mathcal{H}$. If $(\mathcal{X}, \mathcal{Y})_{K}=\left(\left\{x_i\right\}_{i=1}^{\infty},\left\{y_i\right\}_{i=1}^{\infty}\right)$, $(\mathcal{Z}, \mathcal{Y})_{K}=\left(\left\{z_i\right\}_{i=1}^{\infty},\left\{y_i\right\}_{i=1}^{\infty}\right)$ and $ (\mathcal{X}, \mathcal{Z})_{K}=\left(\left\{x_i\right\}_{i=1}^{\infty},\left\{z_i\right\}_{i=1}^{\infty}\right)$ are $K$-biframes for $\mathcal{H}$, then  $(\mathcal{Z}+\mathcal{X}, \mathcal{Y}+\mathcal{Z})_{K}=\left(\left\{z_i\right\}_{i=1}^{\infty}+\left\{x_i\right\}_{i=1}^{\infty},\left\{y_i\right\}_{i=1}^{\infty}+\left\{z_i\right\}_{i=1}^{\infty}\right)$ is a $K$-biframes for $\mathcal{H}$. 
\end{proposition} 
\begin{proof}
for every $x\in \mathcal{H}$, we have
$$
\begin{aligned}
\sum_{i=1}^{\infty}\left\langle x,z_i+ x_i\right\rangle\left\langle y_i+z_i, x\right\rangle&=\sum_{i=1}^{\infty}\left(\left\langle x,z_i\right\rangle+\left\langle x,x_i\right\rangle\right )\left(\left\langle y_i, x\right\rangle +\left\langle z_i, x\right\rangle\right)\\
&=\sum_{i=1}^{\infty}\left\langle x,z_i\right\rangle\left\langle y_i, x\right\rangle+\sum_{i=1}^{\infty}\vert\left\langle x,z_i\right\rangle\vert^{2}+\sum_{i=1}^{\infty}\left\langle x,x_i\right\rangle\left\langle y_i, x\right\rangle+\sum_{i=1}^{\infty}\left\langle x,x_i\right\rangle\left\langle z_i, x\right\rangle
\end{aligned}
$$
Since  $(\mathcal{X}, \mathcal{Y})_{K}$, $(\mathcal{Z}, \mathcal{Y})_{K}$ and $ (\mathcal{X}, \mathcal{Z})_{K}$ are $K$-biframes and  $\lbrace z_{i}\rbrace_{i=1}^{\infty}$ be a $K$-frame for $\mathcal{H}$, we have
\begin{equation}\label{eq31}
A_{\mathcal{X}, \mathcal{Y}}\left\|K^{\ast} x\right\|^2 \leq \sum_{i=1}^{\infty}\left\langle x, x_i\right\rangle \left\langle y_i, x\right\rangle \leq B_{\mathcal{X}, \mathcal{Y}}\|x\|^2, \text { for all } x \in \mathcal{H} .
\end{equation}
And 
\begin{equation}\label{eq32}
A_{\mathcal{Z}, \mathcal{Y}}\left\|K^{\ast} x\right\|^2 \leq \sum_{i=1}^{\infty}\left\langle x, z_i\right\rangle \left\langle y_i, x\right\rangle \leq B_{\mathcal{Z}, \mathcal{Y}}\|x\|^2, \text { for all } x \in \mathcal{H} .
\end{equation}
And
\begin{equation}\label{eq33}
A_{\mathcal{X}, \mathcal{Z}}\left\|K^{\ast} x\right\|^2 \leq \sum_{i=1}^{\infty}\left\langle x, x_i\right\rangle \left\langle z_i, x\right\rangle \leq B_{\mathcal{X}, \mathcal{Z}}\|x\|^2, \text { for all } x \in \mathcal{H} .
\end{equation}
 And
\begin{equation}\label{eq34}
A\left\|K^{\ast} x\right\|^2 \leq \sum_{i=1}^{\infty}\vert\left\langle x,z_i\right\rangle\vert^{2} \leq B\|x\|^2, \text { for all } x \in \mathcal{H} .
\end{equation}
Hence From (\ref{eq31}), (\ref{eq32}),  (\ref{eq33}) and (\ref{eq34}) we conclude that for all $x \in \mathcal{H}$,
$$(A_{\mathcal{X}, \mathcal{Y}}+A_{\mathcal{Z}, \mathcal{Y}}+A_{\mathcal{X}, \mathcal{Z}}+A)\left\|K^{\ast} x\right\|^2 \leq\sum_{i=1}^{\infty}\left\langle x,z_i+ x_i\right\rangle\left\langle y_i+z_i, x\right\rangle\leq (B_{\mathcal{X}, \mathcal{Y}}+B_{\mathcal{Z}, \mathcal{Y}}+B_{\mathcal{X}, \mathcal{Z}}+B)\|x\|^2.$$
Which implies that $(\mathcal{Z}+\mathcal{X}, \mathcal{Y}+\mathcal{Z})_{K}$ is a $K$-biframes for $\mathcal{H}$. 
\end{proof} 
\begin{theorem}
Let $K_1, K_2 \in \mathcal{B}(\mathcal{H})$. If  $(\mathcal{X}, \mathcal{Y})_{K}=\left(\left\{x_i\right\}_{i=1}^{\infty},\left\{y_i\right\}_{i=1}^{\infty}\right)$ is an K$_{j}$-biframe for $j\in\lbrace 1,\;2\rbrace$  and $\alpha_{1}, \alpha_{2}$ are scalars. Then the following holds: 
\begin{enumerate}
\item $(\mathcal{X}, \mathcal{Y})_{K}$ is  $(\alpha_{1} K_1+\alpha_{2} K_2)$-biframe
\item $(\mathcal{X}, \mathcal{Y})_{K}$ is  $K_1 K_2$-biframe.
\end{enumerate}
\end{theorem}
\begin{proof}
(1) Let $(\mathcal{X}, \mathcal{Y})_{K}=\left(\left\{x_i\right\}_{i=1}^{\infty},\left\{y_i\right\}_{i=1}^{\infty}\right)$
  be $K_{1}$-biframe and $K_{2}$-biframe. Then for $j=1$, there exist two constants $0<A \leq B<\infty$ such that
$$
A\left\|K_1^{\ast} f\right\|^2 \leq \sum_{i=1}^{\infty}\left\langle x, x_i\right\rangle\left\langle y_i, x\right\rangle \leq B\|x\|^2, \text { for all } x \in \mathcal{H} .
$$
And for $j=2$, there exist two constants $0<C \leq D<\infty$ such that
$$
C\left\|K_2^{\ast} x\right\|^2 \leq \sum_{i=1}^{\infty}\left\langle x, x_i\right\rangle\left\langle y_i, x\right\rangle \leq D\|x\|^2, \text { for all } x \in \mathcal{H} .
$$
Now, we can write
$$
\begin{aligned}
\left\|\left(\alpha_{1} K_1+\alpha_{2} K_2\right)^{\ast} x\right\|^2 & \leq|\alpha_{1}|^2\left\|K_1^{\ast} x\right\|^2+|\alpha_{2}|^2\left\|K_2^{\ast} x\right\|^2 \\
& \leq|\alpha_{1}|^2\left(\frac{1}{A} \sum_{i=1}^{\infty}\left\langle x, x_i\right\rangle\left\langle y_i, x\right\rangle\right)+|\alpha_{2}|^2\left(\frac{1}{C} \sum_{i=1}^{\infty}\left\langle x, x_i\right\rangle\left\langle y_i, x\right\rangle\right) \\
& =\left(\frac{|\alpha_{1}|^2}{A}+\frac{|\alpha_{2}|^2}{C}\right) \sum_{i=1}^{\infty}\left\langle x, x_i\right\rangle\left\langle y_i, x\right\rangle .
\end{aligned}
$$

It follows that
$$
\left(\frac{A C}{C|\alpha|^2+A|\alpha_{2}|^2}\right)\left\|\left(\alpha_{1} K_1+\alpha_{2} K_2\right)^{\ast} x\right\|^2 \leq \sum_{i=1}^{\infty}\left\langle x, x_i\right\rangle\left\langle y_i, x\right\rangle .
$$

Hence $(\mathcal{X}, \mathcal{Y})_{K}$ satisfies the lower frame condition. And we have
$$
\sum_{i=1}^{\infty}\left\langle x, x_i\right\rangle\left\langle y_i, x\right\rangle \leq\min \lbrace B,D\rbrace\|x\|^2, \text { for all } x \in \mathcal{H}
$$
it follows that
$$
\left(\frac{A C}{C|\alpha|^2+A|\alpha_{2}|^2}\right)\left\|\left(\alpha_{1} K_1+\alpha_{2} K_2\right)^{\ast} x\right\|^2\leq\sum_{i=1}^{\infty}\left\langle x, x_i\right\rangle\left\langle y_i, x\right\rangle \leq\min \lbrace B,D\rbrace\|x\|^2, \text { for all } x \in \mathcal{H}
$$
Therefore $(\mathcal{X}, \mathcal{Y})_{K}$ is  $(\alpha_{1} K_1+\alpha_{2} K_2)$-biframe.

(2) Now for each $x \in \mathcal{H}$, we have
$$
\left\|\left(K_1 K_2\right)^{\ast} x\right\|^2=\left\|K_2^{\ast} K_1^{\ast} x\right\|^2 \leq\left\|K_2^{\ast}\right\|^2\left\|K_1^{\ast} x\right\|^2 .
$$

Since  $(\mathcal{X}, \mathcal{Y})_{K}$ is $K_1$-biframe, then there exist two constants $0<A \leq B<\infty$ such that
$$
A\left\|K_1^{\ast} x\right\|^2 \leq \sum_{i=1}^{\infty}\left\langle x, x_i\right\rangle\left\langle y_i, x\right\rangle \leq B\|x\|^2, \text { for all } x \in \mathcal{H} .
$$
Therefore
$$
\frac{1}{\left\|K_2^{\ast}\right\|^2} \left\|\left(K_1 K_2\right)^{\ast} x\right\|^2\leq\left\|K_1^{\ast} x\right\|^2 \leq \frac{1}{A} \sum_{i=1}^{\infty}\left\langle x, x_i\right\rangle\left\langle y_i, x\right\rangle \leq \frac{B}{A}\|x\|^2 .
$$

This implies that
$$
\frac{A}{\left\|K_2^{\ast}\right\|^2}\left\|\left(K_1 K_2\right)^{\ast} x\right\|^2 \leq \sum_{i=1}^{\infty}\left\langle x, x_i\right\rangle\left\langle y_i, x\right\rangle \leq B\|x\|^2, \text { for all } x \in \mathcal{H} .
$$

Therefore $(\mathcal{X}, \mathcal{Y})_{K}$ is $K_1 K_2$-biframe for $\mathcal{H} $.
\end{proof}

\begin{corollary}
 Let $n\in\mathbb{N}\setminus\lbrace  0,1\rbrace$ and $K_j \in \mathcal{B}(\mathcal{H})$ for $j\in  [\![1;n]\!] $. If  $(\mathcal{X}, \mathcal{Y})_{K}=\left(\left\{x_i\right\}_{i=1}^{\infty},\left\{y_i\right\}_{i=1}^{\infty}\right)$ is an K$_{j}$-biframe for $j\in  [\![1;n]\!]$  and $\alpha_{1}, \alpha_{2}\cdots, \alpha_{n}$ are scalars. Then the following holds: 
\begin{enumerate}
\item $(\mathcal{X}, \mathcal{Y})_{K}$ is  $(\sum\limits_{j=1}^{n}\alpha_{j} K_j)$-biframe
\item $(\mathcal{X}, \mathcal{Y})_{K}$ is  $(K_1K_2\cdots K_{n})$-biframe.
\end{enumerate}
\end{corollary}
\begin{proof}
(1) Suppose that $n\in\mathbb{N}\setminus\lbrace  0,1\rbrace$ and for every $j\in  [\![1;n]\!]$, $(\mathcal{X}, \mathcal{Y})_{K}$ is  K$_{j}$-biframe . Then for each $j\in  [\![1;n]\!]$ there exist positive constants $0<A_{j} \leq B_{j}<\infty$ such that
$$
A_{j}\left\|K_j^{\ast} x\right\|^2 \leq \sum_{i=1}^{\infty}\left\langle x, x_i\right\rangle\left\langle y_i, x\right\rangle \leq B_{j}\|x\|^2, \text { for all } x \in \mathcal{H} .
$$
Now, we can write
$$
\begin{aligned}
\left\|\left(\sum\limits_{j=1}^{n}\alpha_{j} K_j\right)^{\ast} x\right\|^2 & =\left\|\alpha_{1} K_1^{\ast} x+\left(\alpha_{2}K_2+\cdots+\alpha_{n} K_n\right)^{\ast} x\right\|^2 \\
& \leq|\alpha_{1}|^2\left\|K_1^{\ast} x\right\|^2+\left\|\left(\alpha_{2}K_2+\cdots+\alpha_{n} K_n\right)^{\ast} x\right\|^2 \\
& \leq|\alpha_{1}|^2\left\|K_1^{\ast} x\right\|^2+\cdots+|\alpha_{n}|^2\left\|K_n^{\ast} x\right\|^2 \\
& \leq|\alpha_{1}|^2\left(\frac{1}{A_{1}} \sum_{i=1}^{\infty}\left\langle x, x_i\right\rangle\left\langle y_i, x\right\rangle\right)+\cdots+|\alpha_{n}|^2\left(\frac{1}{A_{n}} \sum_{i=1}^{\infty}\left\langle x, x_i\right\rangle\left\langle y_i, x\right\rangle\right) \\
& =\left(\frac{|\alpha_{1}|^2}{A_{1}}+\cdots+\frac{|\alpha_{n}|^2}{A_{n}}\right) \sum_{i=1}^{\infty}\left\langle x, x_i\right\rangle\left\langle y_i, x\right\rangle\\
&= \left(\sum_{j=1}^{n}\frac{|\alpha_{j}|^2}{A_{j}}\right) \sum_{i=1}^{\infty}\left\langle x, x_i\right\rangle\left\langle y_i, x\right\rangle .
\end{aligned}
$$
Hence $(\mathcal{X}, \mathcal{Y})_{K}$ satisfies the lower frame condition. And we have
$$
\sum_{i=1}^{\infty}\left\langle x, x_i\right\rangle\left\langle y_i, x\right\rangle \leq\min\limits_{j\in  [\![1;n]\!]} \lbrace B_{j}\rbrace\|x\|^2, \text { for all } x \in \mathcal{H}
.$$
It follows that
$$
\left(\sum_{j=1}^{n}\frac{|\alpha_{j}|^2}{A_{j}}\right)^{-1}\left\|\left(\sum\limits_{j=1}^{n}\alpha_{j} K_j\right)^{\ast} x\right\|^2\leq\sum_{i=1}^{\infty}\left\langle x, x_i\right\rangle\left\langle y_i, x\right\rangle \leq\min\limits_{j\in  [\![1;n]\!]} \lbrace B_{j}\rbrace\|x\|^2, \text { for all } x \in \mathcal{H}.
$$
Hence $(\mathcal{X}, \mathcal{Y})_{K}$ is  $(\sum\limits_{j=1}^{n}\alpha_{j} K_j)$-biframe

(2) Now for each $x \in \mathcal{H}$, we have
$$
\left\|(K_1K_2\cdots K_{n})^{\ast} x\right\|^2=\left\|K_n^{\ast} \cdots K_1^{\ast} x\right\|^2 \leq\left\|K_n^{\ast} \cdots K_2^{\ast}\right\|^2\left\|K_1^{\ast} x\right\|^2 .
$$
Since  $(\mathcal{X}, \mathcal{Y})_{K}$ is $K_1$-biframe, then there exist two constants $0<A_{1} \leq B_{1}<\infty$ such that 
$$
A_{1}\left\|K_1^{\ast} x\right\|^2 \leq \sum_{i=1}^{\infty}\left\langle x, x_i\right\rangle\left\langle y_i, x\right\rangle \leq B_{1}\|x\|^2, \text { for all } x \in \mathcal{H} .
$$
Therefore
$$
\frac{1}{\left\|K_n^{\ast} \cdots K_2^{\ast}\right\|^2} \left\|(K_1K_2\cdots K_{n})^{\ast} x\right\|^2 \leq \frac{1}{A_{1}} \sum_{i=1}^{\infty}\left\langle x, x_i\right\rangle\left\langle y_i, x\right\rangle \leq \frac{B_{1}}{A_{1}}\|x\|^2 .
$$

This implies that
$$
\frac{A_{1}}{\left\|K_n^{\ast} \cdots K_2^{\ast}\right\|^2}\left\|(K_1K_2\cdots K_{n})^{\ast} x\right\|^2\leq \sum_{i=1}^{\infty}\left\langle x, x_i\right\rangle\left\langle y_i, x\right\rangle \leq B_{1}\|x\|^2, \text { for all } x \in \mathcal{H} .
$$

Therefore $(\mathcal{X}, \mathcal{Y})_{K}$ is $(K_1K_2\cdots K_{n})$-biframe for $\mathcal{H} $.
\end{proof}
\begin{theorem}
Let $K \in \mathcal{B}(\mathcal{H})$ with $\|K\| \geq 1$. Then every ordinary biframe is a $K$-biframe for $\mathcal{H}$.
\end{theorem}
\begin{proof}
Suppose that $(\mathcal{X}, \mathcal{Y})_{K}=\left(\left\{x_i\right\}_{i=1}^{\infty},\left\{y_i\right\}_{i=1}^{\infty}\right)$
  is biframe for $\mathcal{H}$. Then there exist two constants $0<A \leq B<\infty$ such that
$$
A\left\| x\right\|^2 \leq \sum_{i=1}^{\infty}\left\langle x, x_i\right\rangle\left\langle y_i, x\right\rangle \leq B\|x\|^2, \text { for all } x \in \mathcal{H} .
$$
For $K \in \mathcal{B}(\mathcal{H})$, we have 
$$\left\|K^{\ast} x\right\|^{2} \leq\|K\|^{2}\|x\|^{2}, \quad \forall x \in \mathcal{H}.$$ Since $\|K\| \geq 1$, we obtain $$\frac{1}{\|K\|^2}\left\|K^{\ast} x\right\|^2 \leq\|x\|^2\quad \forall x \in \mathcal{H}.$$ Therefore
$$
\frac{A}{\|K\|^2}\left\|K^{\ast} x\right\|^2 \leq A\|x\|^2 \leq \sum_{i=1}^{\infty}\left\langle x, x_i\right\rangle\left\langle y_i, x\right\rangle \leq B\|x\|^2, \text { for all } x \in \mathcal{H} \text {. }
$$

Therefore $(\mathcal{X}, \mathcal{Y})_{K}$ is a $K$-biframe for $\mathcal{H}$.
\end{proof}
\begin{theorem}\label{Th3.7}
Let $(\mathcal{X}, \mathcal{Y})_{K}=\left(\left\{x_i\right\}_{i=1}^{\infty},\left\{y_i\right\}_{i=1}^{\infty}\right)$ be a biframe for $\mathcal{H}$. Then 
$ (\mathcal{X}, \mathcal{Y})_{K}$ is a $K$-biframe for $\mathcal{H}$, if and only if there exists $A>0$ such that $S_{\mathcal{X}, \mathcal{Y}} \geq A K K^{\ast}$, where $S_{\mathcal{X}, \mathcal{Y}}$ is the biframe operator for $(\mathcal{X}, \mathcal{Y})_{K}$.
\end{theorem}
\begin{proof}
 $(\mathcal{X}, \mathcal{Y})_{K}$ is a $K$-biframe for $\mathcal{H}$ with frame bounds $A, B$ and biframe operator $S_{\mathcal{X}, \mathcal{Y}} $, if and only if
$$
A\left\|K^{\ast} x\right\|^2 \leq \langle S_{\mathcal{X}, \mathcal{Y}} x, x\rangle =\sum_{n=1}^{\infty}\left\langle x, x_i\right\rangle\left\langle y_i, x\right\rangle\leq B\|x\|^2, \quad \forall x \in \mathcal{H},
$$
that is,
$$
\left\langle A K K^{\ast} x, x\right\rangle \leq\langle S_{\mathcal{X}, \mathcal{Y}} x, x\rangle \leq\langle B x, x\rangle, \quad \forall x \in \mathcal{H} .
$$

So the conclusion holds.
\end{proof}
\begin{theorem}
 Let $(\mathcal{X}, \mathcal{Y})_{K}=\left(\left\{x_i\right\}_{i=1}^{\infty},\left\{y_i\right\}_{i=1}^{\infty}\right)$ be a biframe for $\mathcal{H}$, with biframe operator $S_{\mathcal{X}, \mathcal{Y}}$ which satisfies $S_{\mathcal{X}, \mathcal{Y}}^{\frac{1}{2}^{\ast}}=S_{\mathcal{X}, \mathcal{Y}}^{\frac{1}{2}}$. Then $(\mathcal{X}, \mathcal{Y})_{K}$ is a $K$ biframe for $\mathcal{H}$ if and only if $K=S_{\mathcal{X}, \mathcal{Y}}^{\frac{1}{2}} U$, for some $U \in \mathcal{B}(\mathcal{H})$.
\end{theorem}
\begin{proof}
Assume that  $(\mathcal{X}, \mathcal{Y})_{K}$ is a $K$-biframe, by Theorem \ref{Th3.7},there exists $A>0$ such that
$$
A K K^{\ast} \leq S_{\mathcal{X}, \mathcal{Y}}^{\frac{1}{2}} S_{\mathcal{X}, \mathcal{Y}}^{\frac{1}{2}^{\ast}} .
$$

Then for each $x \in \mathcal{H}$,
$$\left\|K^{\ast} x\right\|^2 \leq \lambda^{-1}\left\|S_{\mathcal{X}, \mathcal{Y}}^{\frac{1}{2}^{\ast}} x\right\|^2.$$ Therefore Theorem \ref{Doglas th}, $K=S_{\mathcal{X}, \mathcal{Y}}^{\frac{1}{2}} U$, for some $U \in \mathcal{B}(\mathcal{H})$.

Conversely, let $K=S_{\mathcal{X}, \mathcal{Y}}^{\frac{1}{2}} W$, for some $W \in \mathcal{B}(\mathcal{H})$. Then by Theorem \ref{Doglas th},there is a positive number $\mu$ such that
$$
\left\|K^{\ast} x\right\| \leq \mu\left\|S_{\mathcal{X}, \mathcal{Y}}^{\frac{1}{2}} x\right\|, \text { for all } x \in \mathcal{H}
$$
which implies that $$
\mu K K^{\ast} \leq S_{\mathcal{X}, \mathcal{Y}}^{\frac{1}{2}} S_{\mathcal{X}, \mathcal{Y}}^{\frac{1}{2}^{\ast}} .
$$ Since $S_{\mathcal{X}, \mathcal{Y}}^{\frac{1}{2}^{\ast}}=S_{\mathcal{X}, \mathcal{Y}}^{\frac{1}{2}}$ Then by Theorem \ref{Doglas th}, $(\mathcal{X}, \mathcal{Y})_{K}$ is a $K$-biframe for $\mathcal{H}$.
\end{proof}
\begin{example}
 Let $\left\{e_i\right\}_{i=1}^{\infty}$ be an orthonormal basis in $\ell_2$. Define operators $\mathcal{T}$ and $K$ on $\ell_2$ as follows: $\mathcal{T} e_i=e{i-1}$ for $i>1$, $\mathcal{T} e_1=0$, and $K e_i=e_{i+1}$. It is evident that $(\left\{\frac{1}{2}K e_i\right\}_{i=1}^{\infty}, \left\{2K e_i\right\}_{i=1}^{\infty})$ is a $K$-biframe for $\ell_2$. Assume that $(\left\{\frac{1}{2}K e_i\right\}_{i=1}^{\infty}, \left\{2K e_i\right\}_{i=1}^{\infty})$ is a $\mathcal{T}$-biframe. Then by Theorem \ref{Th3.7}, there exists $\lambda>0$ such that $K K^* \geq \lambda \mathcal{T} \mathcal{T}^{\ast}$. Consequently, by Theorem \ref{Doglas th}, $\mathcal{R}(\mathcal{T}) \subseteq \mathcal{R}(K)$.  However, this contradicts $R(\mathcal{T}) \nsubseteq \mathcal{R}(K)$ since $e_1 \in R(\mathcal{T})$ but $e_1 \notin \mathcal{R}(K)$.
\end{example}
In the following proposition,  we establish a necessary condition for the operator $\mathcal{T}$ for which $ (\mathcal{X}, \mathcal{Y})_{K}$  will be $\mathcal{T}$-biframe for $\mathcal{H}$.

 \begin{proposition}
   Let $(\mathcal{X}, \mathcal{Y})_{K}=\left(\left\{x_i\right\}_{i=1}^{\infty},\left\{y_i\right\}_{i=1}^{\infty}\right)$ be a $K$-biframe for $\mathcal{H}$. Let $\mathcal{T} \in \mathcal{B}(\mathcal{H})$ with $R(\mathcal{T}) \subseteq$ $\mathcal{R}(K)$. Then $(\mathcal{X}, \mathcal{Y})_{K}$ is a $\mathcal{T}$-biframe for $\mathcal{H}$.
\end{proposition}
\begin{proof}
Suppose that $(\mathcal{X}, \mathcal{Y})_{K}=\left(\left\{x_i\right\}_{i=1}^{\infty},\left\{y_i\right\}_{i=1}^{\infty}\right)$ is a $K$-biframe for $\mathcal{H}$. Then there are positive constants $0<A \leq B<\infty$ such that
$$
A\left\|K^* x\right\|^2 \leq  \sum_{i=1}^{\infty}\left\langle x, x_i\right\rangle\left\langle y_i, x\right\rangle \leq B\|x\|^2, \text { for all } x \in \mathcal{H} .
$$

Since $R(\mathcal{T}) \subseteq \mathcal{R}(K)$, by Theorem \ref{Doglas th}, there exists $\alpha>0$ such that $\mathcal{T} \mathcal{T}^{\ast} \leq \alpha^2 K K^*$.

Hence, 
$$
\frac{A}{\alpha^2}\left\|\mathcal{T}^{\ast} x\right\|^2 \leq A\left\|K^* x\right\|^2 \leq  \sum_{i=1}^{\infty}\left\langle x, x_i\right\rangle\left\langle y_i, x\right\rangle \leq B\|x\|^2, \text { for all } x \in \mathcal{H} .
$$

Hence $(\mathcal{X}, \mathcal{Y})_{K}$ is a $\mathcal{T}$-biframe for $\mathcal{H}$.
\end{proof}

\begin{theorem}
 Let $(\mathcal{X}, \mathcal{Y})_{K}=\left(\left\{x_i\right\}_{i=1}^{\infty},\left\{y_i\right\}_{i=1}^{\infty}\right)$ be a $K$-biframe for $\mathcal{H}$ with biframe operator $S_{\mathcal{X}, \mathcal{Y}}$ and let $\mathcal{T}$ be a positive operator. Then $(\mathcal{X}+\mathcal{T}\mathcal{X}, \mathcal{Y}+\mathcal{T}\mathcal{Y})_{K}=\left(\left\{x_i+\mathcal{T}x_i\right\}_{i=1}^{\infty},\left\{y_i+\mathcal{T}y_i\right\}_{i=1}^{\infty}\right)$ is a $K$-biframe. 
 
 Moreover for any natural number $n,\left(\left\{x_i+\mathcal{T}^{n}x_i\right\}_{i=1}^{\infty},\left\{y_i+\mathcal{T}^{n}y_i\right\}_{i=1}^{\infty}\right)$ is a $K$-biframe for $\mathcal{H}$.
\end{theorem}
\begin{proof}
Suppose that $(\mathcal{X}, \mathcal{Y})_{K}=\left(\left\{x_i\right\}_{i=1}^{\infty},\left\{y_i\right\}_{i=1}^{\infty}\right)$ is a $K$-biframe for $\mathcal{H}$. Then by Theorem \ref{Th3.7} , there exists $m>0$ such that $S_{\mathcal{X}, \mathcal{Y}} \geq m K K^*$. For every $x \in \mathcal{H}$, we have
$$
\begin{aligned}
S_{(\mathcal{X}+\mathcal{T}\mathcal{X}),( \mathcal{Y}+\mathcal{T}\mathcal{Y})}&=\sum_{i=1}^{\infty}\left\langle x,\left(x_i+\mathcal{T} x_i\right)\right\rangle\left(y_i+\mathcal{T} y_i\right)\\
 & =(I+\mathcal{T}) \sum_{i=1}^{\infty}\left\langle x,(I+\mathcal{T}) x_i\right\rangle y_i \\
 &=(I+\mathcal{T}) \sum_{i=1}^{\infty}\left\langle (I+\mathcal{T})^{\ast}x, x_i\right\rangle y_i\\
& =(I+\mathcal{T}) S_{\mathcal{X},\mathcal{Y}}(I+\mathcal{T})^{\ast} x .
\end{aligned}
$$
Hence the frame operator for $(\mathcal{X}+\mathcal{T}X, \mathcal{Y}+\mathcal{T}Y)_{K}$ is $(I+\mathcal{T}) S_{\mathcal{X},\mathcal{Y}}(I+\mathcal{T})^{\ast}$.  
Since $\mathcal{T}$ is positive operator we get,
$$
(I+\mathcal{T})  S_{\mathcal{X},\mathcal{Y}}(I+\mathcal{T})^{\ast}= S_{\mathcal{X},\mathcal{Y}}+ S_{\mathcal{X},\mathcal{Y}} \mathcal{T}^{\ast}+\mathcal{T}  S_{\mathcal{X},\mathcal{Y}}+\mathcal{T}  S_{\mathcal{X},\mathcal{Y}} \mathcal{T}^{\ast} \geq  S_{\mathcal{X},\mathcal{Y}} \geq m K K^*,
$$
Once again, applying Theorem \ref{Th3.7}, we can conclude that $(\mathcal{X}+\mathcal{T}\mathcal{X}, \mathcal{Y}+\mathcal{T}\mathcal{Y})_{K}$ is a $K$-biframe for $\mathcal{H}$.

Now, for any natural number $n$, the frame operator for
 $$ S_{(\mathcal{X}+\mathcal{T}^{n}\mathcal{X}),( \mathcal{Y}+\mathcal{T}^{n}\mathcal{Y})}=\left(I+\mathcal{T}^{n}\right) S_{\mathcal{X},\mathcal{Y}}(I+\left.\mathcal{T}^{n}\right)^{\ast} \geq S_{\mathcal{X},\mathcal{Y}}. $$ Hence $n,\left(\left\{x_i+\mathcal{T}^{n}x_i\right\}_{i=1}^{\infty},\left\{y_i+\mathcal{T}^{n}y_i\right\}_{i=1}^{\infty}\right)$ is a $K$-biframe for $\mathcal{H}$.
\end{proof}
\section{Operators on $K$-biframes for $\mathcal{H}$}
\begin{example}  Let $\mathcal{H}=\mathbb{C}^4$ and $\left\{e_1, e_2, e_3, e_4\right\}$ be an orthonormal basis for $\mathcal{H}$. Define $K: \mathcal{H} \rightarrow \mathcal{H}$ by $K e_1=e_1, K e_2=e_1, K e_3=e_2, K e_4=e_3$. We consider two sequences $\mathcal{X}=\left\{x_i\right\}_{i=1}^{4}$ and $\mathcal{Y}=\left\{y_i\right\}_{i=1}^{4}$ defined as follows: 
 $$\mathcal{X}=\left\{e_1, e_1, 2e_2,3e_3\right\}$$
 And
  $$\mathcal{Y}=\left\{e_1, e_1, \frac{1}{2}e_2,\frac{1}{3}e_3\right\}.$$
  Then $(\mathcal{X}, \mathcal{Y})_{K}$  is a $K$-biframe for $\mathcal{H}$ with the frame operator $$ S_{\mathcal{X},\mathcal{Y}}=\left(\begin{array}{llll}2 & 0 & 0 & 0 \\ 0 & 1 & 0 & 0 \\ 0 & 0 & 1& 0 \\ 0 & 0 & 0& 0\end{array}\right).$$ This operator is not invertible, (because $\det (S_{\mathcal{X},\mathcal{Y}})=0$).
\end{example}
\begin{theorem}
 Suppose $K \in \mathcal{B}(\mathcal{H})$ has a closed range. The biframe operator of a $K$-biframe is invertible on the subspace $\mathcal{R}(K)$ of $\mathcal{H}$.
\end{theorem}
\begin{proof}
 Assume that $(\mathcal{X}, \mathcal{Y})_{K}=\left(\left\{x_i\right\}_{i=1}^{\infty},\left\{y_i\right\}_{i=1}^{\infty}\right)$ be a $K$-biframe for $\mathcal{H}$.  Then there are positive constants $0<A \leq B<\infty$ such that
$$
A\left\|K^* x\right\|^2 \leq  \sum_{i=1}^{\infty}\left\langle x, x_i\right\rangle\left\langle y_i, x\right\rangle \leq B\|x\|^2, \text { for all } x \in \mathcal{H} .
$$

Since $\mathcal{R}(K)$ is closed, then $K K^{+} x=x$, for all $x \in \mathcal{R}(K)$. That is,
$$
\left.K K^{+}\right|_{\mathcal{R}(K)}=I_{\mathcal{R}(K)},
$$
we have $I_{\mathcal{R}(K)}^{\ast}=\left(\left.K^{+}\right|_{\mathcal{R}(K)}\right)^{\ast} K^*$.
For any $x \in \mathcal{R}(K)$, we obtain
$$
\|x\|^{2}=\left\|\left(\left.K K^{+}\right|_{\mathcal{R}(K)}\right)^{\ast} x \right\|^{2}=\left\|\left(\left.K^{+}\right|_{\mathcal{R}(K)}\right)^{\ast} K^* x\right\| \leq\left\|K^{+}\right\|^{2} \cdot\left\|K^* x\right\|^{2}.
$$
Thus $$\left\|K^* x\right\|^2 \geq\left\|K^{+}\right\|^{-2}\|x\|^2.$$ So, we have
$$
\sum_{i=1}^{\infty}\left\langle x, x_i\right\rangle\left\langle y_i, x\right\rangle \geq A\left\|K^* x\right\|^2 \geq A\left\|K^{+}\right\|^{-2}\|x\|^2, \text { for all } x \in \mathcal{R}(K) .
$$

Therefore, based on the definition of a $K$-biframe, we have
$$
A\left\|K^{+}\right\|^{-2}\|x\|^2 \leq \sum_{i=1}^{\infty}\left\langle x, x_i\right\rangle\left\langle y_i, x\right\rangle \leq B\|x\|^2, \text { for all } x \in \mathcal{R}(K) .
$$

Hence
$$
A\left\|K^{+}\right\|^{-2}\|x\| \leq\|\left.S_{\mathcal{X},\mathcal{Y}}\right|_{\mathcal{R}(K)} x\| \leq B\|x\|, \text { for all } x \in \mathcal{R}(K) .
$$

Consequently, $\left.S_{\mathcal{X},\mathcal{Y}}\right|_{\mathcal{R}(K)}: \mathcal{R}(K) \rightarrow R(S)$ is a bounded linear operator and invertible on $\mathcal{R}(K)$.
\end{proof}
\begin{theorem}
 Let $K \in \mathcal{B}(\mathcal{H})$ be with a dense range. Suppose that $(\mathcal{X}, \mathcal{Y})_{K}=\left(\left\{x_i\right\}_{i=1}^{\infty},\left\{y_i\right\}_{i=1}^{\infty}\right)$ be a $K$-biframe and $\mathcal{T} \in \mathcal{B}(\mathcal{H})$ have a closed range. If $\left(\left\{\mathcal{T}x_i\right\}_{i=1}^{\infty},\left\{\mathcal{T}y_i\right\}_{i=1}^{\infty}\right)$ is a $K$-biframe for $\mathcal{H}$, then $\mathcal{T}$ is surjective.
\end{theorem}
\begin{proof}
Suppose That $\left(\left\{\mathcal{T}x_i\right\}_{i=1}^{\infty},\left\{\mathcal{T}y_i\right\}_{i=1}^{\infty}\right)$ is a $K$-biframe for $\mathcal{H}$ with frame bounds $A$ and $B$. Then for each $x \in \mathcal{H}$,
$$
A\left\|K^* x\right\|^2 \leq  \sum_{i=1}^{\infty}\left\langle x,\mathcal{T} x_i\right\rangle\left\langle \mathcal{T}y_i, x\right\rangle  \leq B\|x\|^2 .
$$
Hence 
\begin{equation}\label{EQ41}
A\left\|K^* x\right\|^2 \leq  \sum_{i=1}^{\infty}\left\langle \mathcal{T}^{\ast}x, x_i\right\rangle\left\langle y_i, \mathcal{T}^{\ast}x\right\rangle  \leq B\|x\|^2 .
\end{equation}
Since $K$ has a dense range, $\mathcal{H}=\overline{\mathcal{R}(K)}$, so $K^*$ is injective. From (\ref{EQ41}), $\mathcal{T}^{\ast}$ is injective since $N\left(\mathcal{T}^{\ast}\right) \subseteq \textbf{N}\left(K^*\right)$. Moreover, $$\mathcal{R}(\mathcal{T})=\mathcal{N}\left(T^*\right)^{\perp}=\mathcal{H}.$$ Therefore $\mathcal{T}$ is surjective.
\end{proof}
\begin{theorem}
 Let $K \in \mathcal{B}(\mathcal{H})$ and let  $(\mathcal{X}, \mathcal{Y})_{K}=\left(\left\{x_i\right\}_{i=1}^{\infty},\left\{y_i\right\}_{i=1}^{\infty}\right)$ be a $K$-biframe for $\mathcal{H}$. If $\mathcal{T} \in \mathcal{B}(\mathcal{H})$ has a closed range with $\mathcal{T} K=K \mathcal{T}$,  then $\left(\left\{\mathcal{T}x_i\right\}_{i=1}^{\infty},\left\{\mathcal{T}y_i\right\}_{i=1}^{\infty}\right)$ is a $K$-biframe for $\mathcal{R}(\mathcal{T})$.
\end{theorem}
\begin{proof}
 Since $\mathcal{T}\in \mathcal{B}(\mathcal{H})$ has a closed range. Then for each $x \in \mathcal{R}(\mathcal{T})$, $$K^* x=\left(\mathcal{T}^{+}\right)^{\ast} \mathcal{T}^{\ast} K^* x,$$ so we have
$$
\left\|K^* x\right\|=\left\|\left(\mathcal{T}^{+}\right)^{\ast} \mathcal{T}^{\ast} K^* x\right\| \leq\left\|\left(\mathcal{T}^{+}\right)^{\ast}\right\|\left\|\mathcal{T}^{\ast} K^* x\right\| .
$$

Hence $$\left\|\left(\mathcal{T}^{+}\right)^{\ast}\right\|^{-1}\left\|K^* x\right\| \leq\left\|\mathcal{T}^{\ast} K^* x\right\|.$$
 Since $(\mathcal{X}, \mathcal{Y})_{K}$ is a $K$-biframe with frame bounds $A, B$, then for each $x \in \mathcal{R}(\mathcal{T})$, we have
$$
\begin{aligned}
 \sum_{i=1}^{\infty}\left\langle x,\mathcal{T} x_i\right\rangle\left\langle \mathcal{T}y_i, x\right\rangle & =\sum_{i=1}^{\infty}\left\langle \mathcal{T}^{\ast}x, x_i\right\rangle\left\langle y_i, \mathcal{T}^{\ast}x\right\rangle \\ &\geq A\left\|K^* \mathcal{T}^{\ast} x\right\|^2 \\
& =A\left\|\mathcal{T}^{\ast} K^* x\right\|^2 \\
& \geq A\left\|\left(\mathcal{T}^{+}\right)^{\ast}\right\|^{-2}\left\|K^* x\right\|^2 .
\end{aligned}
$$
On the other hand, we have
$$
\sum_{i=1}^{\infty}\left\langle x,\mathcal{T} x_i\right\rangle\left\langle \mathcal{T}y_i, x\right\rangle  =\sum_{i=1}^{\infty}\left\langle \mathcal{T}^{\ast}x, x_i\right\rangle\left\langle y_i, \mathcal{T}^{\ast}x\right\rangle\leq \mu\left\|T^* x\right\|^2 \leq \mu\|T\|^2\|x\|^2 .
$$
Hence $$A\left\|\left(\mathcal{T}^{+}\right)^{\ast}\right\|^{-2}\left\|K^* x\right\|^2\leq\sum_{i=1}^{\infty}\left\langle x,\mathcal{T} x_i\right\rangle\left\langle \mathcal{T}y_i, x\right\rangle  \leq B\|\mathcal{T}\|^2\|x\|^2 .$$ 
Therefore $\left(\left\{\mathcal{T}x_i\right\}_{i=1}^{\infty},\left\{\mathcal{T}y_i\right\}_{i=1}^{\infty}\right)$ is a $K$-biframe for $\mathcal{R}(\mathcal{T})$.
\end{proof}
\begin{theorem}
Let $K \in \mathcal{B}(\mathcal{H})$ be with a dense range. Let  $(\mathcal{X}, \mathcal{Y})_{K}=\left(\left\{x_i\right\}_{i=1}^{\infty},\left\{y_i\right\}_{i=1}^{\infty}\right)$ be a $K$-biframe and suppose $\mathcal{T} \in \mathcal{B}(\mathcal{H})$ have a closed range. If $\left(\left\{\mathcal{T}x_i\right\}_{i=1}^{\infty},\left\{\mathcal{T}y_i\right\}_{i=1}^{\infty}\right)$ and $\left(\left\{\mathcal{T}^{\ast}x_i\right\}_{i=1}^{\infty},\left\{\mathcal{T}^{\ast}y_i\right\}_{i=1}^{\infty}\right)$ are $K$-biframes then $\mathcal{T}$ is invertible.
\end{theorem}
\begin{proof}
 Assume that $\left(\left\{\mathcal{T}x_i\right\}_{i=1}^{\infty},\left\{\mathcal{T}y_i\right\}_{i=1}^{\infty}\right)$  is a $K$-biframe for $\mathcal{H}$ with frame bounds $A_1$ and $B_1$. Then for every $x\in \mathcal{H}$,
\begin{equation}\label{eq42}
A_1\left\|K^* x\right\|^2 \leq \sum_{i=1}^{\infty}\left\langle x,\mathcal{T} x_i\right\rangle\left\langle \mathcal{T}y_i, x\right\rangle \leq B_1\|x\|^2 .
\end{equation}
Since $K$ has a dense range, $K^*$ is injective. Then from (\ref{eq42}), it follows that $T^*$ is injective as $N\left(T^*\right) \subseteq N\left(K^*\right)$. Moreover 
$$\mathcal{R}(\mathcal{T})=\mathcal{N}\left(\mathcal{T}^{\ast}\right)^{\perp}=\mathcal{H}.$$ 
Then $\mathcal{T}$ is surjective.

Suppose that $\left(\left\{\mathcal{T}^{\ast}x_i\right\}_{i=1}^{\infty},\left\{\mathcal{T}^{\ast}y_i\right\}_{i=1}^{\infty}\right)$  is a $K$-biframe for $\mathcal{H}$ with frame bounds $A_2$ and $B_2$. Then for every $x\in \mathcal{H}$,
\begin{equation}\label{eq43}
A_2\left\|K^* x\right\|^2 \leq \sum_{i=1}^{\infty}\left\langle x,\mathcal{T}^{\ast} x_i\right\rangle\left\langle \mathcal{T}^{\ast}y_i, x\right\rangle \leq B_2\|x\|^2 .
\end{equation}

Since $K$ has a dense range, then $K^*$ is injective. From (\ref{eq43}), $\mathcal{T}$ is injective since $\mathcal{N}(\mathcal{T}) \subseteq \mathcal{N}\left(K^*\right)$.

 we can conclude that $\mathcal{T}$ is bijective. Therefore $\mathcal{T}$ is invertible.
\end{proof}
\begin{theorem}
 Let $K \in \mathcal{B}(\mathcal{H})$  be with a dense range. Let  $(\mathcal{X}, \mathcal{Y})_{K}=\left(\left\{x_i\right\}_{i=1}^{\infty},\left\{y_i\right\}_{i=1}^{\infty}\right)$ be a $K$-biframe for $\mathcal{H}$ and let $\mathcal{T}\in \mathcal{B}(\mathcal{H})$ be co-isometry with $\mathcal{T} K=K \mathcal{T}$. Then $\left(\left\{\mathcal{T}x_i\right\}_{i=1}^{\infty},\left\{\mathcal{T}y_i\right\}_{i=1}^{\infty}\right)$ is a $K$-biframe for $\mathcal{H}$.
\end{theorem}
\begin{proof}
Assume that $(\mathcal{X}, \mathcal{Y})_{K}$ be a $K$-biframe for $\mathcal{H}$.  Then there are positive constants $0<A \leq B<\infty$ such that
$$
A\left\|K^* x\right\|^2 \leq  \sum_{i=1}^{\infty}\left\langle x, x_i\right\rangle\left\langle y_i, x\right\rangle \leq B\|x\|^2, \text { for all } x \in \mathcal{H} .
$$
Suppose that $\mathcal{T}\in \mathcal{B}(\mathcal{H})$ be co-isometry with $\mathcal{T} K=K \mathcal{T}$. Then for each $x\in \mathcal{H}$, we have
$$
\begin{aligned}
\sum_{i=1}^{\infty}\left\langle x,\mathcal{T} x_i\right\rangle\left\langle \mathcal{T}y_i, x\right\rangle  &=\sum_{i=1}^{\infty}\left\langle \mathcal{T}^{\ast}x, x_i\right\rangle\left\langle y_i, \mathcal{T}^{\ast}x\right\rangle\\ &\geq A\left\|K^*\mathcal{T}^{\ast} x\right\|^2 \\
& =A\left\|\mathcal{T}^{\ast} K^* x\right\|^2 \\
& =A\left\|K^* x\right\|^2. 
\end{aligned}
$$
On the other hand, for $x \in \mathcal{H}$ all we have 
$$\sum_{i=1}^{\infty}\left\langle x,\mathcal{T} x_i\right\rangle\left\langle \mathcal{T}y_i, x\right\rangle =\sum_{i=1}^{\infty}\left\langle \mathcal{T}^{\ast}x, x_i\right\rangle\left\langle y_i, \mathcal{T}^{\ast}x\right\rangle\leq B\|\mathcal{T}^{\ast}x\|^2 \leq B\|\mathcal{T}\|^2\|x\|^2. $$
Therefore $\left(\left\{\mathcal{T}x_i\right\}_{i=1}^{\infty},\left\{\mathcal{T}y_i\right\}_{i=1}^{\infty}\right)$ is a $K$-biframe for $\mathcal{H}$.
\end{proof}

\medskip

\section*{Declarations}

\medskip

\noindent \textbf{Availablity of data and materials}\newline
\noindent Not applicable.

\medskip

\noindent \textbf{Human and animal rights}\newline
\noindent We would like to mention that this article does not contain any studies
with animals and does not involve any studies over human being.

\medskip

\noindent \textbf{Conflict of interest}\newline
\noindent The authors declare that they have no competing interests.

\medskip

\noindent \textbf{Fundings} \newline
\noindent The authors declare that there is no funding available for this paper.

\medskip

\noindent \textbf{Authors' contributions}\newline
\noindent The authors equally conceived of the study, participated in its
design and coordination, drafted the manuscript, participated in the
sequence alignment, and read and approved the final manuscript.


\medskip


\begin{thebibliography}{0}

\bibitem{Abramovich}Y. A. Abramovich, Charalambos, D. Aliprantis, An invitation to operator
theory, American Mathematical Society, 2002.

\bibitem{Christensen} O. Christensen, An Introduction to Frames and Riesz Bases. Birkhäuser, Basel (2003)

\bibitem{Daubechies2} I. Daubechies, Ten Lectures on Wavelets. SIAM, Philadelphia (1992)
\bibitem{Daubechies}
I. Daubechies, A. Grossmann, Y. Mayer,
\emph{Painless nonorthogonal expansions}, Journal of Mathematical Physics 27 (5) (1986) 1271-1283.
\bibitem{Douglas} R. G. Douglas, On majorization, factorization, and range inclusion of operators on Hilbert space. Proc. Am. Math. Soc. 17, 413–415 (1966).
\bibitem{Duffin}R. J. Duffin, A. C. Schaeffer,
\emph{A class of nonharmonic Fourier series}, Trans. Amer. Math. Soc., 72, (1952), 341-366.
 \bibitem{Fer} A. Fereydooni, A. Safapour, Pair frames, Results Math., 66 (2014) 247–263.

\bibitem{Gabor} D. Gabor, 1946. Theory of communications. J. Inst. Electr. Eng. 93: 429–457.
\bibitem{Han} D. Han, D.R. Larson, Frames, bases, and group representations. Mem. Am. Math. Soc. 147 (2000)

\bibitem{Limaye} B. V. Limaye, Functional analysis, New Age International Publishers Limited, New Delhi, second edition (1996).

\bibitem{MF}M. F. Parizi, A. Alijani and M. A. Dehghan
\emph{Biframes and some their properties}, Journal of Inequalities and Applications, https://doi.org/10.1186/s13660-022-02844-7.





\end{thebibliography}
\end{document}